\documentclass[12pt,oneside,reqno]{scrartcl}

\usepackage{cihan}
\usepackage[affil-it]{authblk}

\usepackage{mathtools}
\usepackage{amsfonts}
\usepackage{tikz}
\usetikzlibrary{positioning}

\usepackage{tikz-cd}

\author{C\.{i}han Bahran}
\affil{School of Mathematics\\
University of Minnesota\\
Minneapolis, MN 55455, USA}
\title{Categorifying induction formulae via divergent series}

\newcommand{\scr}{\mathscr}
\newcommand{\bs}{\backslash}

\usepackage{relsize}	

\usepackage{bbm}
\newcommand{\one}{\mathbbm{1}}

\newcommand {\malg}[1]{\bb{M}_{#1}(\qq)}

\newcommand {\molg}[2]{\bb{M}_{#1}(#2)}
\newcommand {\ilg}[1]{e{\bb{M}_{#1}(\qq)}e}
\newcommand {\bilg}[2]{e{\bb{M}_{#1}(#2)}e}

\newcommand {\ce} {\mathsf{C}}
\newcommand {\de} {\mathsf{D}}
\newcommand {\sce} {\scr{C}}

\newcommand {\mob}{M{\"o}bius\,\,}

\newcommand {\surj} {\twoheadrightarrow}

\newcommand{\lset}[1]{#1\mhyphen\mathsf{set}}

\newcommand{\Lef}{\Lambda}
\newcommand{\LefS}{\Lambda_{\Sigma}}
\newcommand{\rlefs}{\wt{\Lambda}_{\Sigma}}
\newcommand{\eus}{\chi_{\Sigma}}
\newcommand{\reus}{\wt{\chi}_{\Sigma}}
\newcommand{\reu}{\wt{\chi}}

\newcommand{\lefm}{\LL}

\newcommand{\compl}[1]{{#1}_{p}^{\wedge}}
\newcommand{\complet}[1]{{\left( #1 \right)}_{p}^{\wedge}}

\newcommand{\gro}[2]{\int_{#2} \! #1}

\newlength{\bibitemsep}
\newlength{\bibparskip}
\let\oldthebibliography\thebibliography
\renewcommand\thebibliography[1]{%
  \oldthebibliography{#1}%
  \setlength{\parskip}{\bibitemsep}%
  \setlength{\itemsep}{\bibparskip}%
}

\DeclareMathOperator{\res}{res}
\DeclareMathOperator{\ind}{ind}

\DeclareMathOperator{\co}{H}
\DeclareMathOperator{\cfy}{B\!}

\newcommand{\prim}[1]{\mathcal{P}({#1})}

\newcommand{\all}[1]{\mathcal{S}({#1})}

\expandafter\def\expandafter\normalsize\expandafter{%
    \normalsize
    \setlength\abovedisplayskip{8pt}
    \setlength\belowdisplayskip{8pt}
    \setlength\abovedisplayshortskip{5pt}
    \setlength\belowdisplayshortskip{5pt}
}

\hypersetup{
  colorlinks=true,
  citecolor=purple,
  linkcolor=blue
}

\makeatletter
\def\blfootnote{\gdef\@thefnmark{}\@footnotetext}
\makeatother

\begin{document}

\date{}
\maketitle
\blfootnote{\textup{2010} \textit{Mathematics Subject Classification}.
Primary 19A22; Secondary 20C05, 20J06, 55R35.} 
\blfootnote{\textit{Key words and phrases}. Burnside ring, induction theory, stable splittings.}

\begin{onecolabstract}
 We show how to get explicit induction formulae for finite group representations, and more generally for rational Green functors, by summing a divergent series over Dwyer's subgroup and centralizer decomposition spaces. This results in formulae with rational coefficients. The former space yields a well-known induction formula, the latter yields a new one. As essentially immediate corollaries of the existing literature, we get similar formulae in group cohomology and stable splittings of classifying spaces.
\end{onecolabstract}

{\tableofcontents}

\section{Introduction} \label{intro}
Let $G$ be a fixed finite group throughout. An \textbf{induction formula} for a complex character $\chi$ of $G$ is the existence of 
\begin{birki}
 \item some characters $\eta_{H}$ for various (preferably proper!) subgroups $H$ of $G$,
 \item some scalars $\lambda_{H} \in \qq$, and
 \item an equation with induced characters of the form $\chi = \sum_{H} \lambda_{H} \ind_{H}^{G}(\eta_{H})$.
\end{birki}
A classical result of Artin \cite[page 293]{artin} (or see Benson \cite[Theorem 5.6.1]{benson}) says that such a formula always exists where $H$ ranges over the \textbf{cyclic subgroups} of $G$. This then allows Artin to reduce certain arguments from a general finite group to a cyclic one. Brauer later gave an explicit version for Artin's induction theorem:

\begin{thm}[{\cite[Satz 1]{brauer}}\footnote{This theorem is \textbf{not} what is usually meant by ``Brauer's induction theorem'' \cite[Theorem 5.6.4]{benson}, which has integral coefficients.}]  \label{brauer-explicit}
 Writing $\sce$ for the set of \textbf{cyclic subgroups} of $G$ and $\one_{H}$ for the trivial character of $H$, we have
\begin{align*}
 \one_{G} = \sum_{H \in \sce} 
 \frac{-\mu_{\sce_{+}}(H,\infty)}{|G:H|} \ind_{H}^{G}(\one_{H}) \, .
\end{align*}
Here, the poset $\sce_{+}$ is given by adding a unique maximum element $\infty$ to $\sce$, and $\mu_{\sce_{+}}$ is its \mob function.
\end{thm}
\begin{rem}
  Although Theorem \ref{brauer-explicit} is on the surface only an induction formula for the trivial character, a similar formula for an arbitrary character $\chi$ can be obtained immediately by multiplying both sides of the equality with $\chi$ and using Frobenius reciprocity.
\end{rem}
Of course, being over cyclic groups, the \mob coefficients in Theorem \ref{brauer-explicit} can be expressed in terms of the number-theoretic \mob function, but it is the formula we present that generalizes. The generalization of Artin's induction theorem to other rings was obtained by Dress \cite[Theorem 1', Theorem 2]{dress-on-integral}, succeeding Conlon \cite[Corollary 4.6]{conlon} who treated the local case. Later, Webb \cite[Theorem D']{webb-local} found a way to make these existence theorems explicit (as Brauer did for Artin) and obtained a formula which has exactly the \textbf{same coefficients} as Brauer's formula but with a \textbf{larger set $\sce$ of subgroups}, whose size depends, not surprisingly, on how many primes divide the order $|G|$ and remain non-invertible in $R$.

In this paper, we  give a meaning to the right hand side of Brauer's (and Webb's) formula as an entity of its own, for \textbf{any} set $\sce$ of subgroups of $G$ which is closed under conjugation. We emphasize that the coefficients in this formula are usually not integers. To that end, we write $\Omega_{\zz}(G)$ for the \textbf{integral Burnside ring} of $G$, and extend the scalars $\Omega(G) := \qq \otimes_{\zz} \Omega_{\zz}(G)$ to allow rational coefficients. We refer the reader to Benson's book \cite[Section 5.4]{benson} for an introduction to the Burnside ring. We write $[G/H]$ for the equivalence class of the transitive left $G$-set $G/H$ as an element of $\Omega(G)$.

We shall define, as an equivariant generalization of Berger--Leinster's \cite{berger} notion of \textbf{series Euler characteristic} $\eus$, a partial assignment 
\begin{align*}
 \LefS \colon \{\text{finite categories with a $G$-action}\} \dashrightarrow \Omega(G)\, .
\end{align*}
 The subscript $\Sigma$ in $\eus$ and $\LefS$ is there to indicate that a divergent series is involved in the definition, coming from the fact that the nerve of most finite categories have cells in arbitrarily high dimensions, due to loops. If $\LefS(\de)$ is defined for a $G$-category $\de$, we call it the \textbf{series Lefschetz invariant} of $\de$, to keep consistent notation with Th{\'e}venaz \cite{thev} (he defines $\Lambda(\mathsf{P})$ for a finite $G$-poset $\mathsf{P}$ and calls it the Lefschetz invariant) and other papers that build on his work. 
 
 Given any set $\sce$ of subgroups of $G$, Dwyer \cite{dw1,dw} introduced a $G$-category $E \oo_{\sce}$ (to be defined in Section \ref{sub-cen}) for obtaining so-called \textbf{subgroup decompositions} in group (co)homology.  We call $E \oo_{\sce}$ the \textbf{subgroup decomposition category} of $\sce$. For the reader familiar with Dwyer's work, we are using Grodal's notation \cite{grodal-higher, grodal-smith} for this category here instead of Dwyer's $\mathbf{X}_{\sce}^{\beta}$. We compute its series Lefschetz invariant:
\begin{thmx} \label{lef-intro-EOC}
 Let $\sce$ be any set of subgroups of $G$ closed under conjugation. The $G$-category $E \oo_{\sce}$ has series Lefschetz invariant 
\begin{align*}
 \LefS(E \oo_{\sce}) = \sum_{H \in \sce} \frac{-\mu_{\sce_{+}}(H,\infty)}{|G:H|} [G/H]  \in \Omega(G) \, .
\end{align*}
Here, the poset $\sce_{+}$ is given by adding a unique maximum element $\infty$ to $\sce$, and $\mu_{\sce_{+}}$ is its \mob function.
\end{thmx}

To state Webb's result (it generalizes Brauer's) precisely, let us introduce some notation. Given a commutative ring $R$, we write $A_{R}(G)$ for the \textbf{rational representation ring} of $G$ over $R$ (see Section \ref{conseq} for a definition). For any group $H$ and a prime $p$, we write $O_{p}(H)$ for the \textbf{largest} normal $p$-subgroup of $H$. 

\begin{thm} [{\cite[Theorem D']{webb-local}} for $R = \compl{\zz}$] \label{webb-formula}
 Let $R$ be a unital commutative ring. Suppose $\sce$ is a set of subgroups of $G$ closed under conjugation which satisfies the following: 
\begin{birki}
 \item Every cyclic subgroup is in $\sce$.
 \item If $H$ is a subgroup such that $H/O_{p}(H)$ is cyclic for some prime $p$ with $pR \neq R$, then $H \in \sce$.\footnote{Note that the existence of a  prime $p$ with $pR \neq R$ renders condition (1) superfluous.}
\end{birki}
Then the trivial representation $R$ can be written as
\begin{align*}
R = \sum_{H \in \sce} 
 \frac{-\mu_{\sce_{+}}(H,\infty)}{|G:H|} \ind_{H}^{G}(R) \in A_{R}(G) \, .
\end{align*}
Here, the poset $\sce_{+}$ is given by adding a unique maximum element $\infty$ to $\sce$, and $\mu_{\sce_{+}}$ is its \mob function.
\end{thm}


In more elementary terms, Theorem \ref{webb-formula} will yield a formula such as $U = \frac{1}{2}V - \frac{1}{2}W$, where $U,V,W$ are certain $RG$-modules. This means that $U \oplus U \oplus W$ is \textbf{stably isomorphic} with $V$, that is, there exists another $RG$-module $N$ such that $U \oplus U \oplus W \oplus N \cong V \oplus N$ as $RG$-modules. Of course the extra $N$ will be unnecessary if finitely generated $RG$-modules have a cancellative property such as being Krull--Schmidt.

We will show that Webb's (hence also Brauer's) explicit formula can be deduced by linearizing Theorem \ref{lef-intro-EOC}. In this sense the coefficients involved in the formula ``come from'' the category $E \oo_{\sce}$, which may be regarded as an instance of categorification.

In his work, Dwyer \cite{dw1,dw} defined another $G$-category $E \A_{\sce}$, this time for obtaining so-called \textbf{centralizer decompositions} in group (co)homology. Thus we call $E \A_{\sce}$ the \textbf{centralizer decomposition category} of $\sce$. As with the subgroup decomposition case, with the centralizer decomposition category we are using Grodal's notation \cite{grodal-higher, grodal-smith} instead of Dwyer's $\mathbf{X}_{\sce}^{\alpha}$. We compute its Lefschetz invariant, which aptly involves centralizer subgroups.

\begin{thmx} \label{lef-intro-EAC}
 Let $\sce$ be any set of subgroups of $G$ closed under conjugation. The $G$-category $E \A_{\sce}$ satisfies 
\begin{align*}
 \LefS(E \A_{\sce}) = \sum_{H \in \sce} \frac{-\mu_{\sce_{-}}(-\infty,H)}{|G:C_{G}(H)|} [G/C_{G}(H)]  \in \Omega(G) \, .
\end{align*}
Here, the poset $\sce_{-}$ is given by adding a unique minimum element $-\infty$ to $\sce$, and $\mu_{\sce_{-}}$ is its \mob function.
\end{thmx}

As an application, the expansion for $E \A_{\sce}$ in Theorem \ref{lef-intro-EAC}  linearizes into an induction formula which appears to be new:

\begin{thmx} \label{new-induction}
 Let $R$ be a unital commutative ring. Suppose $\sce$ is a set of subgroups of $G$ closed under conjugation which satisfies the following: 
\begin{birki}
 \item If $K$ is a cyclic subgroup, the centralizer $C_{G}(K)$ is in  $\sce$.
 \item If $K$ is a subgroup such that $K/O_{p}(K)$ is cyclic for some prime $p$ with $pR \neq R$, then the centralizer $C_{G}(K)$ is in $\sce$.
\end{birki}
Then the trivial representation $R$ can be written as
\begin{align*}
R = \sum_{H \in \sce} 
 \frac{-\mu_{\sce_{-}}(-\infty,H)}{|G:C_{G}(H)|} \ind_{C_{G}(H)}^{G}(R) \in A_{R}(G) \, .
\end{align*}
Here, the poset $\sce_{-}$ is given by adding a unique minimum element $-\infty$ to $\sce$, and $\mu_{\sce_{-}}$ is its \mob function.
\end{thmx}

An induction formula for group (co)homology immediately follows from Theorem \ref{new-induction} by applying an $\Ext$ or $\Tor$, similar to {\cite[Theorem D]{webb-local}. Here $A_{R}(1)$ is simply the Grothendieck group of finitely generated $R$-modules under direct sum, extended to $\qq$-coefficients.
 
\begin{thmbis}{new-induction} \label{new-coho}
 Let $\sce$ and $R$ be as in the hypotheses of Theorem \ref{new-induction}. Fix a cohomological degree $k \geq 0$, and a finitely generated $RG$-module $M$. We have 
\begin{align*}
 \co^{k}(G;M) = \sum_{H \in \sce} 
 \frac{-\mu_{\sce_{-}}(-\infty,H)}{|G:C_{G}(H)|} \co^{k}(C_{G}(H); M) \in A_{R}(1) \, .
\end{align*}
A similar statement holds for homology $\co_{k}(G;M)$ and Tate cohomology $\hat{\co}^{k}(G;M)$.
\end{thmbis}

Webb's formula can be made even more general, where the representation ring  is replaced by an arbitrary \textbf{rational Green functor}, see Section \ref{green-section} and more specifically Theorem \ref{EOC-induction}. There is an analog of Theorem \ref{new-induction} in the Green functor generality as well: Theorem \ref{EAC-induction}.

On the topology side, Minami showed that \cite[Theorem 6.6]{minami-splitting} Webb's formulae in cohomology can be lifted to suspension spectra of classifying spaces after $p$-completing. Minami's general setup allows us to deduce a similar lift with the centralizer decomposition:

\begin{thmbis}{new-coho} \label{bg-split-cent}
 Let $p$ be a fixed prime. Suppose $\sce$ is a set of subgroups of $G$ closed under conjugation, such that the centralizer $C_{G}(K)$ is in $\sce$ whenever $K/O_{p}(K)$ is cyclic. Then, writing $\compl{X}$ for the $p$-completion of a space $X$, there is a formal stable equivalence 
\begin{align*}
 \compl{\cfy G} \simeq \bigvee_{H \in \sce} \frac{-\mu_{\sce_{-}}(-\infty,H)}{|G:C_{G}(H)|} \compl{\cfy C_{G}(H)} \, .
\end{align*}
of spectra, with respect to the wedge sum $\vee$.
\end{thmbis}

With the words \emph{formal stable equivalence} above, we mean that after clearing the denominators and transferring the negative terms to the left, the  genuine spaces on both sides have homotopy equivalent suspension spectra.

\subsection{Outline}
Below is a graph of logical dependencies among the main theorems of this paper. To highlight the analogies, we include some of the previously known results like Webb's formulae in this graph, distinguishing the results of this paper by \textbf{bold} font.\\

{\small
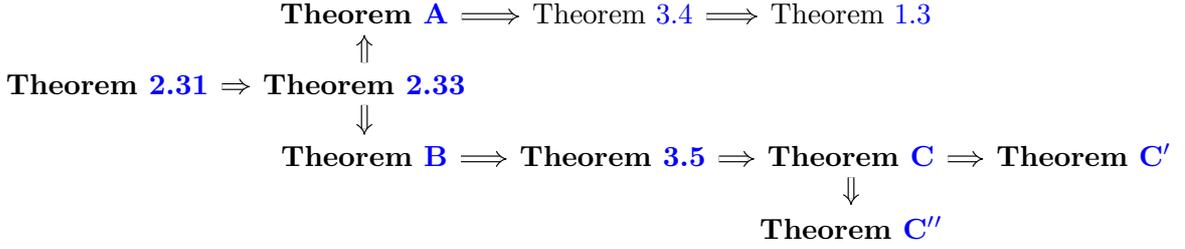
\begin{tikzcd}[column sep = 2.3ex, row sep = 2.3ex]
 & \textbf{Theorem \ref{lef-intro-EOC}} 
 \arrow[Rightarrow]{r}
 & \text{Theorem \ref{EOC-induction}}
 \arrow[Rightarrow]{r}
 & \text{Theorem \ref{webb-formula}}
\\
 \textbf{Theorem \ref{compute}} 
 \arrow[Rightarrow]{r}
 & \textbf{Theorem \ref{leflef}}
 \arrow[Rightarrow]{u}
 \arrow[Rightarrow]{d} \\
 & \textbf{Theorem \ref{lef-intro-EAC}}
 \arrow[Rightarrow]{r}
 & \textbf{Theorem \ref{EAC-induction}}
 \arrow[Rightarrow]{r}
 & \textbf{Theorem \ref{new-induction}}
 \arrow[Rightarrow]{r}
 \arrow[Rightarrow]{d}
 & \textbf{Theorem \ref{new-coho}} \\
 & & &\textbf{Theorem \ref{bg-split-cent}}
\end{tikzcd}
}

Theorem \ref{compute} is in a sense the master theorem here. It has three notions involved in it: skeletal weighting of a category $\ce$, the Grothendieck construction $\gro{}{\ce}$ of a functor, and the series Lefschetz invariant $\LefS$ of a $G$-category. The definitions of and the relationships between these three notions is essentially what Section \ref{bulk} is about. Skeletal weighting is obtained from what we call \textbf{skeletal \mob inversion}, introduced in Section \ref{ske}. Skeletal \mob inversion is more of an auxilliary tool, which gives a way to perform Leinster's \cite{leinster} (ordinary) \mob inversion without the need to pass to a skeleton. We recall the skeletal weighting computations of Jacobsen--M{\o}ller \cite{jacobsen} for the orbit and fusion categories associated to a set of subgroups $\sce$ in Section \ref{subgroup}.  The series Lefschetz invariant $\LefS(\de)$ of a $G$-category $\de$, is a direct adaptation of the \textbf{series Euler characteristic} $\eus$ of Berger--Leinster \cite{berger} to the equivariant context. We review the series Euler characteristic in Section \ref{series-Euler} and define the series Lefschetz invariant in Section \ref{section-lefschetz}. 

The Grothendieck construction is a general way of gluing different categories together. We review it both in the non-equivariant and the equivariant contexts in Sections \ref{gro-non-eq} and \ref{gro-eq}. The main categories of interest in this paper, $E \oo_{\sce}$ and $E \A_{\sce}$, are both obtained as Grothendieck constructions. After proving Theorem \ref{compute} which tells us how to compute $\LefS$ of a general Grothendieck construction, we use the computations of Jacobsen--M{\o}ller \cite{jacobsen} (recalled in Section \ref{subgroup}) to compute $\LefS(E \oo_{\sce})$ and $\LefS(E \A_{\sce})$ in Theorem \ref{leflef}.

Having defined and computed $\LefS(E \oo_{\sce})$ and $\LefS(E \A_{\sce})$ in the rational Burnside ring $\Omega(G)$, Section \ref{green-section} proceeds in a rather formal fashion by pushing them any $\qq$-Green functor, culminating in the explicit induction formulae: Theorem \ref{EOC-induction} and Theorem \ref{EAC-induction}. The final section (Section \ref{canon}) addresses the \textbf{canonicity} of the induction formulae, in the sense of Boltje \cite{boltje-ja}. It has no bearing on our main results and can be safely skipped in a first reading.
\subsection{Related work} 
The divergent series for Euler characteristic type alternating sums come about for the categories we are interested in because their \textbf{nerves} are infinite-dimensional cell-complexes. On the other hand, there are several results in the literature which yield induction theorems in group theory by putting a \textbf{finite} $G$-complex $X$ into work. Without divergent summations like $\sum_{n} (-1)^{n} = \frac{1}{2}$ that arise for infinite-dimensional spaces, this approach naturally results in \textbf{integral} coefficients. In this case, one usually writes $\Lambda(X) \in \Omega(G)$ for the finite alternating sum (the more classical Lefschetz invariant \cite{thev}) and its linearization $\lefm(X)$ for the \textbf{Lefschetz module}.  Here is a sampling for previous work in this vein: 
\begin{birki}
 \item Snaith \cite{snaith-88} gave a categorification of Brauer's induction theorem \cite[Theorem 5.6.4]{benson}. Snaith takes $X$ to be a certain quotient of unitary matrices $U(n)$ with $n = \dim_{\cc}(V)$, which has a translation $G$-action by a defining homomorphism $\rho_{V}: G \rarr U(n)$ of $V$. The discussion through the vanishing of the Lefschetz module appears explicitly in \cite[2.10(d)]{snaith-app}. \\
 
 \item Symonds \cite[\S 2]{symonds-91} gave a different categorification of Brauer induction. For a $G$-module $V$, he takes $X = \mathbb{P}(V)$ , the projective space on $V$, and a twisted version of the Lefschetz module using the tautological line bundle. The formula Symonds gets is indeed different than Snaith's, and the two are compared in Boltje--Snaith--Symonds \cite{bss}.\\
 
 \item Fix a prime number $p$ and write $\compl{\zz}$ for the $p$-adic integers. Webb \cite{webb-local} takes $X$ to be either the order complex $\mathcal{S}_{p}(G)$ of the poset of non-identity $p$-subgroups (the Brown complex), or more generally any $G$-complex with certain fixed point conditions. He shows that the \textbf{reduced} Lefschetz module is a virtual projective \cite[Theorem A']{webb-local} $\compl{\zz}G$-module. This can be seen as an induction theorem in the stable sense, which is a formula that holds ``modulo projectives''. Because the (Tate)-cohomology of a  projective module vanishes, an induction theorem for group cohomology \cite[Theorem A]{webb-local} follows. Webb revisited these results later in two ways. First, he showed that the two induction formulae are actually equivalent to each other \cite[Main Theorem]{webb-greening}. And second, he refined them  into a structure theorem about the augmented chain complex $\wt{C}_{*}(X; \compl{\zz})$ in \cite[Theorem 2.7.1]{webb-split}.\\
 
 \item A surprising theorem of Bouc \cite[Theorem 1.1]{bouc-contract} says that it is enough for $X$ to be \textbf{non-equivariantly} contractible as a space for $\wt{C}_{*}(X;R)$ to be \textbf{equivariantly} chain homotopy equivalent to the zero complex, regardless of what the commutative ring $R$ is. That $X$ is a finite complex is a crucial assumption here, through use of Smith theory. Kropholler--Wall \cite[Section 5]{krop-wall} observed that using Bouc's theorem together with Oliver's classification \cite{oliver-fixed} of the class of finite groups which can act on a contractible complex with no fixed points, one obtains Dress's induction theorem \cite[page 47, Proposition 9.4]{dress-K}.\\
\end{birki}

It is also imperative to mention the work of Grodal \cite{grodal-higher} and Villarroel-Flores--Webb \cite{flores} which work with the same categories that we do. In these papers, the infinite-dimensionality of $E \oo_{\sce}$ and $E \A_{\sce}$ is dealt with by separating the isomorphisms from the non-isomorphisms. The isomorphisms in these categories all come from conjugations in $G$, whereas the non-isomorphisms basically yield $\sce$ itself as a poset, whose order complex is of course finite-dimensional. For both $E \oo_{\sce}$ and $E \A_{\sce}$, the main induction statement of these papers is the existence of a \textbf{finite} split exact chain complex \cite[Theorem 1.4, Corollary 8.13-14]{grodal-higher}, \cite[Main Theorem]{flores} involving group (co)homology, when the set of subgroups $\sce$ is large enough. As a result the Lefschetz module of these chain complexes vanish, resulting in induction formulae for group (co)homology. These formulae are different than ours. Most importantly, they are integral and they involve (co)invariants.

\subsection*{Acknowledgements} I wish to thank my advisor Peter J. Webb for his invaluable suggestions and guidance. I also thank the anonymous referee for inquiring about Minami type decompositions.

\section{\mob inversion, Euler characteristic, Lefschetz invariant} \label{bulk}
\subsection{Skeletal \mob Inversion} \label{ske}

We extend Leinster's notion of \mob inversion \cite{leinster} in a category, to incorporate isomorphisms. We call this procedure \textbf{skeletal \mob inversion}. Nonskeletal categories are not in any way an obstruction for Leinster's theory of Euler characteristic, because one can always pass to a skeleton. But the algebra of skeletal \mob inversion makes certain computations go through more easily.

\begin{conv}
 Throughout this paper, $\mathsf{C}$ is assumed to be a \textbf{finite} category: $\ce$ has finitely many objects, and the set $\mathsf{C}(x,y)$ of morphisms between any two objects $x,y$ is finite.
\end{conv}

\begin{defn}[{\cite[1.1]{leinster}}] \label{mob-alg}
 We denote by $\malg{\mathsf{C}}$ the $\qq$-algebra of functions $\Obj \ce \times \Obj \ce \rarr \qq$ with pointwise addition and scalar multiplication, multiplication defined by
\begin{align*}
 \alpha \beta(x,y) = \sum_{z \in \Obj \ce} \alpha(x,z)\beta(z,y) \, . 
\end{align*}
Similarly we define $\molg{\ce}{R}$ for any commutative ring $R$, considering $R$-valued functions.
\end{defn}
The Kronecker delta $\delta$ is the multiplicative identity of $\malg{\ce}$. The \textbf{zeta function} $\zeta_{\ce} \in \malg{\ce}$ is defined by $\zeta_{\ce}(x,y) := |\ce(x,y)|$. If $\zeta_{\ce}$ is invertible in $\malg{\ce}$, then $\ce$ is said to have \textbf{M{\"o}bius inversion}, and $\mu_{\ce} := \zeta_{\ce}^{-1}$ is called the \textbf{M{\"o}bius function} of $\ce$.

\begin{rem}
 If $\ce$ is a finite poset considered as a finite category in the standard way, then $\zeta_{\ce}$ is guaranteed to be invertible (see Example \ref{ex-ei} for a generalization). In this case the $\mu_{\ce}$ defined above is nothing but the classical \mob function \cite[Section 3.7]{stanley-enum-1}, \cite[Example 1.2.a]{leinster} for the poset. The $\mu_{\sce_{+}}$ and $\mu_{\sce_{{-}}}$ mentioned in the theorems of Section \ref{intro} are defined in this way.
\end{rem}

We begin setting the stage for skeletal \mob inversion. Write $[x]_{\ce}$ or shortly $[x]$ for the isomorphism class of an object $x$ in $\ce$, so that $|[x]|$ is the \textbf{size} of the isomorphism class. Now define 
\begin{align*}
 e_{\ce} \colon \Obj \ce \times \Obj \ce &\rarr \qq \\
 (x,y) &\mapsto 
\begin{cases}
 \mathlarger{\frac{1}{|[x]|}} & \text{if $x,y$ are isomorphic in $\ce$,} \\
 0 & \text{otherwise.}
\end{cases}
\end{align*}
We claim that $e = e_{\ce} \in \malg{\ce}$ is an idempotent. Indeed,
\begin{align*}
 e^{2}(x,y) &= \sum_{z \in \Obj \ce} e(x,z)e(z,y) 
 = \frac{1}{|[x]|} \sum_{z \cong x} e(z,y)
\end{align*}
is \ds{\frac{1}{|[x]|}} if $x$ and $y$ are isomorphic, and zero otherwise.

We consider the $\qq$-algebra $\ilg{\ce}$, whose multiplicative identity is $e_{\ce}$. Note that for $\alpha \in \malg{\ce}$,
\begin{align*}
 e\alpha e(x,y) = \sum_{z,t \in \Obj \ce} e(x,z) \alpha(z,t) e(t,y) 
 = \frac{1}{|[x]|\cdot|[y]|}\sum_{z \in [x], t \in [y]} \alpha(z,t) \, ;
\end{align*}
so the linear map $\malg{\ce} \rarr \ilg{\ce}$ given by $\alpha \mapsto e \alpha e$ is a kind of averaging operation on the isomorphism classes of $\ce$. This yields the following characterization for $\ilg{\ce}$:
\begin{prop}
 A function $\alpha \in \malg{\ce}$ lies in $\ilg{\ce}$ if and only if $\alpha$ is invariant under isomorphisms; that is, $\alpha(x,y) = \alpha(x',y')$ whenever $x \cong x', y \cong y'$.
\end{prop}
In particular, the zeta function $\zeta_{\ce}$ is always in $\ilg{\ce}$.

\begin{defn}
 The category $\ce$ is said to have \textbf{skeletal \mob inversion} if $\zeta_{\ce}$ has an inverse in $\ilg{\ce}$, in which case we denote the inverse by $\nu_{\ce} = \nu \in \ilg{\ce}$. 
\end{defn}

\begin{prop} \label{skeletal}
The following are equivalent: 
\begin{birki}
 \item $\ce$ has skeletal \mob inversion.
 \item $\ce$ has a skeleton $[\ce]$ with (ordinary) \mob inversion.
 \item There exists $\beta \in \malg{\ce}$ such that $\zeta_{\ce} \beta = e_{\ce}$.
 \item There exists $\alpha \in \malg{\ce}$ such that $\alpha \zeta_{\ce} = e_{\ce}$.
 \item There exist $\alpha, \beta \in \malg{\ce}$ such that $\alpha \zeta_{\ce} \beta = e_{\ce}$.
\end{birki}
\end{prop}
\begin{proof}
 To see (1) $\Leftrightarrow$ (2), pick any skeleton $[\ce]$ of $\ce$, and consider the $\qq$-linear isomorphism 
\begin{align*}
 \ilg{\ce} &\rarr \malg{[\ce]} \\
 \alpha &\mapsto \alpha^{*}
\end{align*}
given by composing the inclusion $\ilg{\ce} \emb \malg{\ce}$ with the restriction $\malg{\ce} \surj \malg{[\ce]}$. We see that $e^{*} \in \malg{[\ce]}$ is invertible, and a straightforward computation shows that for every $\alpha,\beta \in \ilg{\ce}$ we have $(\alpha \beta)^{*} = \alpha^{*} (e^{*})^{-1} \beta^{*}$. Therefore $\alpha$ is invertible in $\ilg{\ce}$ if and only if $\alpha^{*}$ is invertible in $\malg{[\ce]}$ with $(\alpha^{*})^{-1} = (e^{*})^{-1} (\alpha^{-1})^{*} (e^{*})^{-1}$. In particular, $\zeta_{\ce} \in \ilg{\ce}$ is invertible if and only if $(\zeta_{\ce})^{*} = \zeta_{[\ce]} \in \malg{[\ce]}$ is invertible. The rest of the equivalences follow from basic linear algebra.
\end{proof}

\begin{ex} \label{ex-ei}
 There is a wide class of finite categories with skeletal \mob inversion called \textbf{EI-categories}; that is, categories in which every endomorphism is an isomorphism. To see this, suppose $\ce$ is an EI-category. A skeleton of $\ce$ is still EI, hence by Proposition \ref{skeletal}(2) we may assume $\ce$ is skeletal and show $\ce$ has \mob inversion. In this case the \emph{a priori} preorder on $\Obj(\ce)$ defined by $x \leq y \iff \ce(x,y) \neq \empt$ is actually a partial order. Extend this partial order $\leq$ to a linear order on $\Obj(\ce)$. With this ordering, we may regard $\malg{\ce}$ as a matrix algebra, in which $\zeta_{\ce}$ corresponds to an upper triangular matrix with nonzero diagonal (because of identity morphisms). Thus $\zeta_{\ce} \in \malg{\ce}$ is invertible.
\end{ex}

\begin{rem} \label{module}
 For any commutative ring $R$, the set $R^{\Obj(\ce)}$ of functions from $\Obj(\ce)$ to $R$ with pointwise addition and scalar multiplication is a left (resp. right) $\molg{\ce}{R}$-module via 
\begin{align*}
 (\alpha f)(x) := \sum_{y \in \Obj(\ce)} \alpha(x,y)f(y) \, , \quad  \text{resp. }
 (f \beta)(y) := \sum_{x \in \Obj(\ce)}f(x)\beta(x,y)  \, .
\end{align*}
There is nothing fancy going on here. Once we put an ordering on $\Obj(\ce)$, what we have described is just the left and right action of the matrix algebra on the set of column and row vectors, respectively. We just do not commit to such an ordering as the expressions are cleaner with the indexing given by the objects themselves. However, in a concrete example, putting an ordering and proceeding with good old matrices is the most efficient way to do calculations.
\end{rem}
We have
\begin{align*}
 e_{\ce}{\qq}^{\Obj(\ce)} = \qq^{\Obj(\ce)}e_{\ce} = \{f \colon \Obj(\ce) \rarr \qq: f(x)=f(y) \text{ whenever } x \cong y\} \, ,
\end{align*}
which has both left and right $\ilg{\ce}$-module structures via restricting from $\malg{\ce}$.

\begin{defn}[{\cite[1.10, 2.1, 2.2]{leinster}}]  \label{defin}
 Write $\one \in \qq^{\Obj(\ce)}$ for the function that sends every object of $\ce$ to $1$. A function $k \in \qq^{\Obj(\ce)}$ is called a \textbf{weighting} on $\ce$ if $\zeta_{\ce} k = \one$, and a \textbf{coweighting} if $k \zeta_{\ce} = \one$. If $\ce$ has both a weighting $k$ and a coweighting $k'$, then the common value 
\begin{align*}
 \chi({\ce}) := \sum_{x \in \Obj(\ce)} k(x) = \sum_{x \in \Obj(\ce)} k'(x) \in \qq
\end{align*}
is called the \textbf{Euler characteristic} of $\ce$, and $\reu(\ce) := \chi(\ce) -1$ is called the \textbf{reduced Euler characteristic} of $\ce$.
\end{defn}

\begin{rem} \label{opposite}
 Write $\ce^{\opp}$ for the opposite category of $\ce$. Then $k \in \Obj(\ce)^{\qq} = \Obj(\ce^{\opp})^{\qq}$ is a weighting of $\ce$ if and only if it is a coweighting of $\ce^{\opp}$.
\end{rem}

If $\ce$ has (ordinary) \mob inversion $\mu$, the functions $\mu \one$ and $\one \mu$ are the unique weightings and coweightings on $\ce$, respectively. Also the sum of the values of $\mu$ equals $\chi(\ce)$ \cite[page 32]{leinster}. Proposition \ref{weight} and Corollary \ref{skeletal-euler} are generalizations of these facts to the case when $\ce$ has skeletal \mob inversion, replacing $\mu$ with $\nu$.

\begin{prop} \label{weight}
Suppose $\ce$ has skeletal \mob inversion, such that $\alpha, \beta \in \malg{\ce}$ satisfy $\zeta_{\ce} \beta = e$ and $\alpha \zeta_{\ce} = e$. Then $\beta\one$ is a weighting on $\ce$, and $\one \alpha$ is a coweighting on $\ce$. In particular, $\nu_{\ce} \one$ (resp. $\one \nu_{\ce}$) is  the unique weighting (resp. coweighting) on $\ce$ that is constant on the isomorphism classes of $\Obj(\ce)$.
\end{prop}

\begin{proof}
Note that $\one \in e\qq^{\Obj(\ce)}$; so $\zeta (\beta \one) = (\zeta \beta) \one = e \one = \one$ via the left $\malg{\ce}$-module structure on $\qq^{\Obj(\ce)} $. Similarly, $(\one \alpha) \zeta = \one$.  The uniqueness claim follows from $\zeta \in \ilg{\ce}$ acting invertibly on $e\qq^{\Obj(\ce)} = \qq^{\Obj(\ce)}e$ from both sides.
\end{proof}

\begin{cor} \label{skeletal-euler}
 If $\ce$ has skeletal \mob inversion, then $\ce$ has Euler characteristic $\chi(\ce) = \sum_{x,y \in \Obj(\ce)} \nu_{\ce}(x,y)$.
\end{cor}

\begin{defn} \label{skeletal-w}
 If $\ce$ has skeletal \mob inversion, we call the unique (co)weighting that is constant on the isomorphism classes the \textbf{skeletal (co)weighting} of $\ce$. 
\end{defn}
Note that the skeletal (co)weighting of $\ce$ can also be obtained via distributing the unique (co)weighting of a skeleton $[\ce]$ uniformly among the isomorphism classes of objects.
 
\subsection{Grothendieck construction (non-equivariant)} \label{gro-non-eq}
This is a very important construction for us that we will come back to again.
\begin{defn} \label{groth-defn}
  Given any functor $F \colon \ce \rarr \Cat$, the \textbf{Grothendieck construction} $\gro{F}{\ce}$ is a category defined as follows:
\begin{itemize}
 \item $\Obj(\gro{F}{\ce}) = \{(x,a): x \in \Obj(\mathsf{C}), a \in \Obj(F(x))\}$,
 \item $\gro{F}{\ce}((x,a),(y,b)) = \{(\alpha, u): \,\,\, \alpha: x \rarr y \text{ in } \ce, u:F(\alpha)(a) \rarr b \text{ in } F(y)\}$,
\end{itemize}
with composition defined in the natural way: 
$
 (\beta,v) \cdot (\alpha,u) := (\beta \alpha, v \cdot F(\beta)(u))
$.
\end{defn}

The Grothendieck construction is significant in homotopy theory, due to a theorem of Thomason \cite[Theorem 1.2]{tomas} which identifies $\gro{F}{\ce}$ as the homotopy colimit of $F$. Its Euler characteristic is the weighted sum of pointwise Euler characteristics under $F$:

\begin{prop} [{\cite[Proposition 2.8]{leinster}}] \label{elements}
Let $k \colon \Obj(\ce) \rarr \qq$  be a weighting on $\ce$ and suppose that $F \colon \ce \rarr \Cat$  is a functor such that $\gro{F}{\ce}$ and each $F(x)$ have Euler characteristics. Then
\begin{align*}
 \chi \left( \gro{F}{\ce} \right) = \sum_{x \in \Obj(\ce)} k(x) \chi(F(x)) \, .
\end{align*}
\end{prop}
We will prove an equivariant version of this weighted sum formula in Theorem \ref{compute}. All of the formulae in this paper will follow from it.
 
\subsection{Orbit and fusion categories} \label{subgroup}
In this section, we note the (co)weightings in the orbit and fusion categories  associated to a finite group $G$ and a set of subgroups $\sce$. The weights for the orbit category will be precisely the coefficients in Theorem \ref{lef-intro-EOC} and the coweights for the fusion category will be those of Theorem \ref{lef-intro-EAC}. When $\sce$ consists of $p$-subgroups, these have been worked out in Jacobsen-M{\o}ller \cite{jacobsen}. Most of the results in \cite{jacobsen} generalize to more general collections of subgroups. These can be obtained more systematically from scratch via skeletal \mob inversion, but we shall not do so here.

Let $\sce$ be a set of subgroups of $G$ closed under conjugation. We  write $\sce_{\geq H}$ for the set that consists of subgroups in $\sce$ that contain $H$, regardless of whether $H$ is in $\sce$ or not. We will similarly write $\sce_{<H}$, etc. Note that the subposets $\sce_{\leq H}, \sce_{<H}, \sce_{\geq H}, \sce_{>H}$ no longer have a $G$-action, but an $N_{G}(H)$-action.

There will be two important categories whose sets of objects are both $\sce$. First is the \textbf{orbit category} $\oo_{\sce}$, where $\oo_{\sce}(H,K)$ is the set of $G$-maps from $G/H$ to $G/K$. Note that $\oo_{\sce}$ is an EI-category, hence has skeletal \mob inversion by Example \ref{ex-ei}.

\begin{prop} [{\cite[Theorem 3.3]{jacobsen}}] \label{orbit-weight}
 Let $\sce$ be a set of subgroups of $G$ closed under conjugation. The skeletal weighting of the orbit category $\oo_{\sce}$ is given by
\begin{align*}
 k \colon \sce &\rarr \qq \\
 H &\mapsto \frac{-\mu_{\sce_{+}}(H,\infty)}{|G:H|} = \frac{-\reu(\sce_{>H})}{|G:H|} \, .
\end{align*}
\end{prop}

%

%

The second subgroup category we consider is the \textbf{fusion category} $\F_{\sce}$, whose object set is $\sce$  and $\F_{\sce}(H,K)$ is the set of group homomorphisms from $H$ to $K$ that are induced from conjugation by an element of $G$. The fusion category is also an EI-category, hence has \mob inversion.

\begin{prop} [{\cite[Theorem 3.3]{jacobsen}}] \label{fusion-coweight}
 Let $\sce$ be a set of subgroups of $G$ closed under conjugation. The fusion category $\F_{\sce}$ has skeletal \mob inversion, with skeletal coweighting
\begin{align*}
 t \colon \sce &\rarr \qq \\
 K &\mapsto \frac{-\mu_{\sce_{-}}(-\infty,K)}{|G:C_{G}(K)|} = \frac{-\reu(\sce_{<K})}{|G: C_{G}(K)|} \, .
\end{align*}
\end{prop}

\subsection{Series Euler characteristic} \label{series-Euler}
In this section, we review the notion of \textbf{series Euler characteristic}, due to Berger--Leinster \cite{berger} with our skeletal \mob inversion framework from Section \ref{ske}. This way, we obtain a simpler proof in Corollary \ref{series-ok} of a theorem of Berger--Leinster \cite[Theorem 3.2]{berger} about the coincidence of the series Euler characteristic with Leinster's earlier notion of Euler characteristic (given here in Definition \ref{defin}). In addition, this section will serve as a template and be used itself when we define the equivariant analog of series Euler characteristic in Section \ref{section-lefschetz}.

To any finite category $\ce$, we can associate a simplicial set $N \ce$ via the \textbf{nerve} construction. We can also go further and take the geometric realization of $N \ce$, often denoted by $B \ce := |N \ce|$ and called the \textbf{classifying space} of $\ce$. The classifying space $B \ce$ has a CW-complex structure where $n$-cells are given by the non-degenerate $n$-simplices of $N \ce$, which in turn are given by $n$-tuples of composable morphisms 
\begin{align*} 
 \xymatrix{ x_{0} \ar[r]^{\vphi_{0}} & x_{1} \ar[r]^{\vphi_{1}\,\,\,\,\,\,\,} & \,\,\,\, \cdots \,\,\,\, \ar[r]^{\,\,\,\,\vphi_{n-1}} & x_{n}}
\end{align*} 
in $\ce$ such that none of $\vphi_{i}$ is the identity. In other words, an $n$-cell of $B \ce$ is a path of length $n$ in the underlying graph of $\ce$ such that none of the constituent edges is an identity morphism. Let us write $\ce_{n}$ for the set of all $n$-cells. In particular, $\ce_{0} = \Obj(\ce)$, and $\ce_{1}$ is the set of \textbf{non-identity} morphisms in $\ce$. Note that each $\ce_{n}$ is a finite set because $\ce$ has finitely many morphisms. But  $B \ce$ might have infinitely many cells: this occurs precisely when $\ce$ has non-degenerate cycles. In this case the classical Euler characteristic as an alternating sum of the number of cells is not defined. With the idea of evaluating at $-1$ if possible, we form the formal power series
\begin{align*}
 f_{\ce}(t) := \sum_{n \geq 0} |\ce_{n}|t^{n} \in \zz[[t]] \, .
\end{align*}
We have the following characterization:

\begin{prop}[{\cite[Lemma 1.3, Proposition 2.11]{leinster}}] \label{karakter}
 The following are equivalent: 
\begin{birki}
 \item The series $f_{\ce}$ is actually a polynomial.
 \item There are finitely many cells in $B \ce$.
 \item The category $\ce$ is skeletal and the only endomorphisms in $\ce$ are identities.
\end{birki}
Furthermore, in this case the classical Euler characteristic $\chi(B \ce) := f_{\ce}(-1) \in \zz$ exists and is equal to $\chi(\ce)$ in the sense of Definition \ref{defin}.
\end{prop}
The main idea of Berger--Leinster \cite{berger}, to pursue the alternating sum point of view for the Euler characteristic of a category possibly outside the class characterized in Proposition \ref{karakter}, is that even when $f_{\ce} \in \zz[[t]]$ is not a polynomial (so the alternating sum of $|\ce_{n}|$'s  diverges), it might be a rational function that can be evaluated at $-1$, which in general will give a number in $\qq$ rather than $\zz$.

\begin{defn} \cite[Definition 2.3]{berger} \label{series-euler-defn}
 The category $\ce$ is said to have \textbf{series Euler characteristic} if $f_{\ce}$ lies in the localization 
\begin{align*}
 \qq[t]_{(t+1)} = \zz[t]_{(t+1)} = \left\{\frac{p(t)}{q(t)}: p(t),q(t) \in \zz[t] \text{ and } (t+1) \nmid q(t)\right\}\, ,
\end{align*}
and it is defined by $\chi_{\Sigma}(\ce) := f_{\ce}(-1) \in \qq$.
\end{defn}

\begin{rem}
 The ring $\zz[t]_{(t+1)}$ does not really lie inside $\zz[[t]]$, because $t$ is invertible in the former but not the latter. Still, whether a formal power series lies in $\zz[t]_{(t+1)}$ or not is a well-defined notion, because both rings canonically embed in the ring of Laurent series over $\zz$. More concretely, the subset $\zz[[t]] \cap \zz[t]_{(t+1)} \subseteq \zz[[t]]$  consists of formal power series $f$ such that there exists a polynomial $q \in \zz[t]$ that is not a multiple of $t+1$ which makes $fq$ a polynomial. 
\end{rem}

The elementary but key observation made in \cite[Theorem 2.2]{berger}, to see $f_{\ce}$ is \textbf{always} a rational function, is that writing $\delta \in \malg{\ce}$ for the Kronecker delta (the multiplicative identity of $\malg{\ce}$) and $\zeta = \zeta_{\ce}$ for the zeta function, we have 
\begin{align*}
 (\zeta - \delta)^{n}(x,y) = |\{\text{non-degenerate $n$-simplices in $N \ce$ that start with $x$ and end at $y$}\}| \, ,
\end{align*}
for every $n \geq 0$, because of the way multiplication is defined in $\malg{\ce}$. As a refinement of $f_{\ce}$, we consider the generating function over $\malg{\ce}$:
\begin{align*}
 w_{\ce}(t) := \sum_{n \geq 0} (\zeta - \delta)^{n}t^{n} = \frac{\delta}{\delta - (\zeta - \delta)t} \in \malg{\ce}[[t]] \cong \molg{\ce}{\qq[[t]]}\, ,
\end{align*}
where the equality above follows by the invertibility of geometric series. In particular, $w_{\ce}(t)$ is not just a matrix of power series, but a matrix of rational functions over $\qq$.

To be able to evaluate the rational functions we get at $-1$, they should lie in $\qq[t]_{(t+1)}$. An arbitrary $\ce$ may not satisfy this condition, yet we have the following:

\begin{prop} \label{series-weight}
 Suppose $\ce$ has skeletal \mob inversion. Acting by $w_{\ce} \in \molg{\ce}{\qq(t)}$ on $\one \in \qq(t)^{\Obj(\ce)}$ from the left, the image of the function 
\begin{align*}
 w_{\ce} \one \colon \Obj(\ce) \rarr \qq(t)
\end{align*}
is contained in the localization $\qq[t]_{(t+1)}$, which means we can evaluate  at $-1$ and get a map $w_{\ce} \one (-1): \Obj(\ce) \rarr \qq$.
This map is a weighting on $\ce$. The analogous claims  hold for the coweighting using the right module structure.
\end{prop}
\begin{proof}
 Let us write $e = e_{\ce}$, $w = w_{\ce}$ and $\delta = \delta_{\ce}$. Since $\zeta e = e \zeta = \zeta \in \bilg{\ce}{\qq}$, using the rational function expression of $w$ obtained above, we get 
\begin{align*}
 e &= \delta e =  w (\delta - (\zeta - \delta)t)e = w (e - (\zeta - e)t) = we (e - (\zeta - e)t) 
\end{align*}
in $\bilg{\ce}{\qq(t)}$, hence \ds{we = \frac{e}{e-(\zeta - e)t} \in \bilg{\ce}{\qq(t)}} because $e$ is the multiplicative identity of $\bilg{\ce}{\qq(t)}$. Now, $e - (\zeta - e)t$ is a polynomial that evaluates to $\zeta \in \ilg{\ce}$ when we plug in $t = -1$. But by assumption, $\zeta \in \ilg{\ce}$ is invertible with inverse $\nu$; thus $we \in \bilg{\ce}{\qq[t]_{(t+1)}}$ and $we(-1) = \nu$. Finally, 
\begin{align*}
 (w \cdot \one)(-1) = (w e \cdot \one)(-1) = (we)(-1) \cdot \one = \nu \cdot \one
\end{align*} is a weighting by Proposition \ref{weight}. The coweighting claim follows dually.
\end{proof}

\begin{rem} \label{dikkat}
 By the definition of $w_{\ce}$, we always have
\begin{align*}
 w_{\ce} \one \colon \Obj(\ce) & \rarr \qq(t) \\
 x &\mapsto \sum_{n \geq 0} |\{\text{non-degenerate $n$-simplices in $N \ce$ that start with $x$}\}|t^{n} \, ,
\end{align*}
and
\begin{align*}
 \one w_{\ce} \colon \Obj(\ce) & \rarr \qq(t) \\
 y &\mapsto \sum_{n \geq 0} |\{\text{non-degenerate $n$-simplices in $N \ce$ that end at $y$}\}|t^{n} \, .
\end{align*}
What Proposition \ref{series-weight} says is that when $\ce$ has skeletal \mob inversion, both $w_{\ce} \one$ and $\one w_{\ce}$ can be evaluated at $-1$ to give a weighting and a coweighting on $\ce$, respectively. Actually by  Proposition \ref{weight}, they give the skeletal weighting and coweighting on $\ce$.
\end{rem}

As a result of our setup with skeletal \mob inversion, we can prove Berger--Leinster's main positive result without using transfer matrix method type identities such as \cite[Proposition 2.5]{berger}.

\begin{cor}[{\cite[Theorem 3.2]{berger}}] \label{series-ok}
 If $\ce$ has skeletal \mob inversion, then $\ce$ has both Euler characteristic and series Euler characteristic, and $\chi(\ce) = \chi_{\Sigma}(\ce)$. 
\end{cor}
\begin{proof}
 By Proposition \ref{series-weight}, $(w_{\ce} \one)(-1)$ is a weighting on $\ce$ and $(\one w_{\ce})(-1)$ is a coweighting on $\ce$. Hence $\ce$ has Euler characteristic
\begin{align*}
 \chi(\ce) &= \sum_{x \in \Obj(\ce)} (w_{\ce} \one)(-1)(x) 
 = \left( \sum_{x \in \Obj(\ce)} (w_{\ce}\one)(x) \right)(-1)
 = f_{\ce}(-1) = \chi_{\Sigma}(\ce) \, .
\end{align*}
Here $f_{\ce} \in \qq[t]_{(t+1)}$ because $w_{\ce} \one \in (\qq[t]_{(t+1)})^{\Obj(\ce)}$.
\end{proof}

\subsection{Series Lefschetz invariant} \label{section-lefschetz}

Let $G$ be a finite group and let $\de$ be a finite category with a $G$-action. It is desirable to have a kind an Euler characteristic which, in some sense, remembers the $G$-action that is present. The \textbf{rational} Burnside ring $\Omega(G)$ is a natural home for such an invariant (we refer the reader to Benson \cite[Section 5.4]{benson} for background about the Burnside ring). The power series 
\begin{align*}
 f_{\de}(t) = \sum_{n \geq 0} \de_{n}t^{n}
\end{align*}
lies in $\Omega(G)[[t]]$, because each 
\begin{align*}
 \de_{n} &= \{ \text{the set of non-degenerate $n$-simplices of $N\de$} \} \\
 &= \{\text{$n$-tuples of composable maps in $\de$ without identity arrows}\}
\end{align*}
is now a $G$-set. In case $f_{\de}$ is a polynomial, $\Lef(\de) := f_{\de}(-1)$ actually lies in $\Omega_{\zz}(G)$, and is often called the \textbf{Lefschetz invariant} of $\de$ (or $N \de$), see \cite[Section 1]{thev}, \cite[Section 6]{webb-sub}. Now we can try to play the same game used to define the series Euler characteristic $\chi_{\Sigma}$ here. 

First, observe that the natural $\qq$-algebra morphism
\begin{align*}
\Omega(G) \otimes_{\qq} \qq[[t]] \rarr \Omega(G)[[t]]
\end{align*}
is an isomorphism because $\dim_{\qq} \Omega(G) < \infty$. Now regarding $f_{\de} \in \Omega(G) \otimes_{\qq} \qq[[t]]$, we mimic Definition \ref{series-euler-defn}:

\begin{defn} \label{lefs-defn}
 The $G$-category $\de$ is said to have \textbf{series Lefschetz invariant} if $f_{\de}$ lies in $\Omega(G) \otimes_{\qq} \qq[t]_{(t+1)}$ and it is defined by $\LefS(\de) := f_{\de}(-1) \in \Omega(G)$. If we wish to emphasize the group $G$, we write $\LefS^{(G)}(\de)$.
\end{defn}
In virtually every equivariant situation, the construction for $G$ has an analog for every subgroup $H \leq G$ which talk to each other via restriction and conjugation (and sometimes induction) maps. The series Lefschetz invariant is not different. With the obvious choices for the restriction ($\res_{H}^{K}$ and $\Res_{H}^{K}$) and conjugation ($c_{g}$ and $\mathbf{c}_{g}$) maps and functors, the following is evident:

\begin{prop} \label{lef-respect}
 Whenever $H \leq K$ are subgroups of $G$ and for every $g \in G$, the diagrams 
\begin{align*}
 \xymatrix{
 \{\text{finite $K$-categories}\} \ar@{-->}[r]^-{\LefS} \ar[d]_{\Res_{H}^{K}} & \Omega(K) \ar[d]^{\res_{H}^{K}} \\
 \{\text{finite $H$-categories}\} \ar@{-->}[r]^-{\LefS} & \Omega(H) \\}
\quad \text{and} \quad
 \xymatrix{
 \{\text{finite $H$-categories}\} \ar@{-->}[r]^-{\LefS} \ar[d]_{\mathbf{c}_{g}} & \Omega(H) \ar[d]^{c_{g}} \\
 \{\text{finite $^{g}H$-categories}\} \ar@{-->}[r]^-{\LefS} & \Omega({^{g}H}) \\}
\end{align*}
commute (in the appropriate sense for partially defined functions).
\end{prop}

We finish this section by establishing routine properties of the series Lefschetz invariant $\LefS$. Without the $\Sigma$ subscript, they appear in Th{\'e}venaz's work \cite{thev}.

Writing $\all{G}$ for the set of \textbf{all} subgroups of $G$, and $\eps_{H} \in \Omega(G)$ for the primitive idempotent in $\Omega(G)$ corresponding to the subgroup $H$ (see \cite[page 179]{benson}, where it is denoted $e_{H}$), we have:

\begin{prop} \label{lef-eu}
 The $G$-category $\de$ has series Lefschetz invariant if and only if for every $H \leq G$ the subcategory $\de^{H}$ has series Euler characteristic. Moreover, in this case we have 
\begin{align*} \label{Lef-mark}
 \LefS(\de) = \sum_{H \in [G \bs \all{G}]}\chi_{\Sigma}(\de^{H}) \eps_{H} \, .
\end{align*}
\end{prop}
\begin{proof}
Fix $H \leq G$ and consider the ring homomorphism $m_{H}: \Omega(G) \rarr \qq$ given by $X \mapsto X^{H}$. Now $m_{H}$ extends to a ring homomorphism $m_{H}: \Omega(G) \otimes_{\qq} \qq[[t]] \cong \Omega(G)[[t]] \rarr \qq[[t]]$ given by $X \otimes g(t) \mapsto |X^{H}|g(t)$. We may restrict to
\begin{align*}
 m_{H} \colon \Omega(G) \otimes_{\qq} \qq[t]_{(t+1)} \rarr \qq[t]_{(t+1)}
\end{align*}
so that the evaluating at -1 yields a commutative diagram 
\begin{align*}
 \xymatrix{
 \Omega(G) \otimes_{\qq} \qq[t]_{(t+1)} \ar[r]^{\,\,\,\,\,\,\,m_{H}} \ar[d]_{\id \otimes \text{ev}_{-1}} & \qq[t]_{(t+1)} \ar[d]^{\text{ev}_{-1}} \\
 \Omega(G) \ar[r]^{m_{H}} & \qq \, .
 }
\end{align*}
Thus if $\de$ has series Lefschetz invariant, chasing $f_{\de}$ in the above commutative diagram yields that $\de^{H}$ has series Euler characteristic and that $m_{H}(\LefS(\de)) = \chi_{\Sigma}(\de^{H})$.

Conversely, if $\de^{H}$ has series Euler characteristic for every $H \leq G$, we have
\begin{align*}
 f_{\de} = \sum_{H \in [G \bs \all{G}]} m_{H}(f_{\de}) \cdot \eps_{H} = \sum_{H \in [G \bs \all{G}]} f_{\de^{H}} \, \eps_{H} \in \Omega(G) \otimes_{\qq} \qq[t]_{(t+1)} \, .
\end{align*}
Thus $\de$ has series Lefschetz invariant, with the desired equality coming from evaluating at $-1$ above.
\end{proof}

Similar with $\chi$, there are reduced versions of $\chi_{\Sigma}$ and $\LefS$. We write $\reus(\ce) := \chi_{\Sigma}(\ce) - 1$ for a finite category $\ce$ with series Euler characteristic, and $\rlefs(\de) := \LefS(\de) - [G/G] \in \Omega(G)$ if $\de$ has series Lefschetz invariant.

\begin{cor} \label{red-lef}
 If the $G$-category $\de$ has series Lefschetz invariant, its reduced series Lefschetz invariant is given by 
\begin{align*}
 \rlefs(\de) = \sum_{H \in [G \bs \all{G}]} \reus(\de^{H}) \eps_{H} \, .
\end{align*}
\end{cor}
\begin{proof}
Immediate from Proposition \ref{lef-eu}.
\end{proof}

\subsection{Grothendieck construction (equivariant)} \label{gro-eq}
We wish to compute the series Lefschetz invariants of a class of $G$-categories introduced by Dwyer \cite[3.1]{dw}. Write $\lset{G}$ for the category of \textbf{finite} $G$-sets. Let $\ce$ be a finite category, and let $F: \ce \rarr \lset{G}$ be any functor. Then, regarding sets as discrete categories, we can form the Grothendieck construction $\gro{F}{\ce}$ (Definition \ref{groth-defn}). 

\begin{rem} \label{E-const}
  Our assumption here that $F$ takes values in sets rather than categories simplifies the structure of $\gro{F}{\ce}$ somewhat. In this case, $\gro{F}{\ce}$ has objects $(x,a)$ where $x \in \Obj(\ce)$ and $a \in \Obj(F(x))$, and morphisms  $(\vphi,a): (x,a) \rarr (y,b)$ where $\vphi: x \rarr y$ in $\ce$ and $F(\vphi)(a) = b$. Furthermore, $G$ acts on objects of $\gro{F}{\ce}$ via $g \cdot (x,a) = (x,ga)$ and on morphisms via $g \cdot (\vphi,a) = (\vphi,ga)$, making $\gro{F}{\ce}$ a $G$-category.
\end{rem}

We will first collect some basic properties of $\gro{F}{\ce}$, constructed as above. Given a $G$-category $\de$, let us write $\Iso_{G}(\de)$ for the set of stabilizer subgroups of simplices of $N \de$. That is, 
\begin{align*}
 \Iso_{G}(\de) = \bigcup_{n \geq 0}\{G_{\sigma}: \sigma \in \de_{n}\} \, .
\end{align*}
Note that $\Iso_{G}(\de)$ is a set of subgroups closed under conjugation, for $^{g}{G_{\sigma}} = G_{g \sigma}$.
\begin{prop} \label{groth-genel}
 Let $F: \ce \rarr \lset{G}$ be any functor. Considering the poset $\sce := \Iso_{G}(\gro{F}{\ce})$ of subgroups as a $G$-category, the assignment 
\begin{align*}
 \Theta \colon \gro{F}{\ce} &\rarr \sce \\
 (x,a) &\mapsto G_{a}
\end{align*}
defines a $G$-equivariant functor, which for every subgroup $H \leq G$ restricts to a $N_{G}(H)$-equivariant functor $\Theta^{H}: (\gro{F}{\ce})^{H} \rarr \sce_{\geq H}$ such that 
\begin{birki}
 \item $\Theta^{H}$ is surjective on objects.
 \item If for each $x \in \Obj(\ce)$ the $G$-set $F(x)$ is transitive, then $\gro{F}{\ce}$ is a preorder and $\Theta^{H}$ is faithful.
 \item If, in addition to (2), $F$ is full, then $\Theta^{H}$ is a (non-equivariant) equivalence of categories.
\end{birki}
\end{prop}
\begin{proof}
 To see $\Theta$ does define a functor, we only need to check $G_{a} \subseteq G_{b}$ whenever $(\vphi,a): (x,a) \rarr (y,b)$ is a morphism in $\gro{F}{\ce}$. But this is immediate because $F(\vphi): F(x) \rarr F(y)$ is a $G$-map with $\vphi(a) = b$; so any $g \in G$ fixing $a$ will fix $b$. Furthermore, if $H$ fixes $(x,a)$, by definition we get $G_{a} \geq H$. This verifies that $\Theta$ does restrict to $\Theta^{H}$ as specified. 
 
For (1), first note that by the definition of the $G$-action on $\gro{F}{\ce}$, a simplex
\begin{align*} \xymatrixcolsep{4pc}\xymatrix{
  \sigma \colon (x_{0},a_{0}) \ar[r]^{(\vphi_{0},a_{0})} & (x_{1},a_{1}) \ar[r]^{(\vphi_{1},a_{1})} & \cdots \ar[r]^{(\vphi_{n-1},a_{n-1})\,\,\,} & (x_{n},a_{n})
  }
\end{align*}
in $(\gro{F}{\ce})_{n}$ is fixed by $g \in G$ if and only if $g$ fixes every $a_{i}$. Thus 
\begin{align*}
 G_{\sigma} = \bigcap_{i=0}^{n} G_{a_{i}} = G_{a_{0}} = \Theta(x_{0},a_{0}) \, .
\end{align*}
Next, we observe the $N_{G}(H)$-equivariance of $\Theta^{H}$, from which the $G$-equivariance of $\Theta$ follows by taking $H = 1$: take $n \in N_{G}(H)$, then for $(x,a) \in (\gro{F}{\ce})^{H}$ we have $n \cdot (x,a) = (x,na) \in E \ce^{H}$ and $G_{na} = {^{n} G_{a}} \geq {^{n}H} = H$.

If $F(x)$ is transitive, there can be at most one morphism from $(x,a)$ to $(y,b)$ in $\gro{F}{\ce}$; hence the assumption (2) forces the entire category $\gro{F}{\ce}$, and hence the subcategory $(\gro{F}{\ce})^{H}$ to be a preorder. Any functor out of a preorder is faithful. Finally, suppose furthermore that $F$ is full. With (1) and (2) in place, we only need to show that $\Theta^{H}$ is full. To that end, let $K \leq L$ in $\sce_{\geq H}$. We want to show that this inclusion $K \leq L$ is the image of a morphism in $\gro{F}{\ce}$.  By (1), there exists $(x,a), (y,b) \in \Obj(\gro{F}{\ce})$ such that $K = G_{a}$ and $L = G_{b}$. In particular, since $G_{a}$ and $G_{b}$ contain $H$ we have $(x,a),(y,b) \in (\gro{F}{\ce})^{H}$. Next, as $F(x)$ is assumed to be transitive and $G_{a} \subseteq G_{b}$, 
\begin{align*}
 \lambda \colon F(x) &\rarr F(y) \\
 ga &\mapsto gb
\end{align*}
is a well-defined $G$-map. As $F$ is full, there exists $\vphi: x \rarr y$ such that $\lambda = F(\vphi)$, and $(\vphi,a): (x,a) \rarr (y,b)$ is a morphism in $(\gro{F}{\ce})^{H}$ that is sent to $G_{a} \subseteq G_{b}$ via $\Theta$ as desired.
\end{proof}

\begin{rem} \label{stab-belli}
 Taking $H =1$ in Proposition \ref{groth-genel}(1), we see that $\Theta$ is surjective on objects. This means that $\Iso_{G}(\gro{F}{\ce})$ consists of the stabilizer subgroups that occur in the various $G$-sets $F(x)$, varying $x \in \Obj(\ce)$.
\end{rem}

The following theorem is the backbone for all the formulae in this paper. It is an equivariant version of Proposition \ref{elements}.
\begin{thm} \label{compute}
 Assume $\ce$ is a finite category with skeletal \mob inversion and $F: \ce \rarr \lset{G}$ is a functor. Then writing $[F(x)] \in \Omega(G)$ for the equivalence class of the $G$-set $F(x)$, the $G$-category $\gro{F}{\ce}$ has series Lefschetz invariant given by
\begin{align*}
 \LefS\left(\gro{F}{\ce}\right) = \sum_{x \in \Obj(\ce)} k_{\ce}(x) [F(x)] \in \Omega(G) \, ,
\end{align*}
where $k_{\ce}$ is the skeletal weighting on $\ce$.
\end{thm}
\begin{proof}
 Write $\de := \gro{F}{\ce}$. And for each $x \in \Obj(\ce)$, let 
\begin{align*}
 \de_{n}(x) &:= \{\sigma \in \de_{n}: \text{$\sigma$ starts with $(x,a)$ for some $a \in F(x)$}\} \\
 \ce_{n}(x) &:= \{\tau \in \ce_{n}: \text{$\tau$ starts with $x$}\} \, .
\end{align*}
The natural projection functor $p \colon \de \rarr \ce$ induces a map $p \colon \de_{n}(x) \rarr \ce_{n}(x)$, because $p(f)$ is non-identity if $f$ is non-identity. We also observe that $\de_{n}(x)$ is a $G$-set equipped with a $G$-map $s \colon \de_{n}(x) \rarr F(x)$ which sends $\sigma$ to the $a \in F(x)$ that appears at the start of the chain $\sigma$. Now we see that the $G$-map
\begin{align*}
 \Phi \colon \de_{n}(x) &\rarr \ce_{n}(x) \times F(x) \\
 \sigma &\mapsto (p(\sigma), s(\sigma))
\end{align*}
where $\ce_{n}(x)$ is considered with the trivial $G$-action, is an isomorphism of $G$-sets. In other words, a chain in $\de_{n}(x)$ is uniquely determined by its image in $\ce_{n}(x)$ and the $a \in F(x)$ that occurs in the beginning. Therefore, as an element of the Burnside ring, we have 
\begin{align*}
 \de_{n} &= \sum_{x \in \Obj(\ce)} \de_{n}(x) = \sum_{x \in \Obj(\ce)} |\ce_{n}(x)|[F(x)] \in \Omega(G) \, ,
\end{align*}
and hence
\begin{align*}
 f_{\de} &= \sum_{n \geq 0}\,\, \sum_{x \in \Obj(\ce)} |\ce_{n}(x)|[F(x)] t^{n}  = \sum_{x \in \Obj(\ce)} w_{\ce} \one(x)[F(x)] \, ,
\end{align*}
using Remark \ref{dikkat} and the notation within. By the same Remark, $w_{\ce} \one(x)$ is a rational function in $\qq[t]_{(t+1)}$ that evaluates to $k_{\ce}(x)$ when we plug in $t=-1$. Thus $f_{\de} \in \Omega(G) \otimes \qq[t]_{(t+1)}$, that is, $\de$ has a series Lefschetz invariant $\Lambda_{\Sigma}(\de) = f_{\de}(-1)$ and it is equal to the desired sum.
\end{proof}
\subsection{Subgroup and centralizer decomposition categories} \label{sub-cen}
We again assume $\sce$ is a set of subgroups of $G$ closed under conjugation. There are two settings in which we consider functors of the form $F \colon \ce \rarr \lset{G}$: 
\begin{birki}
 \item Take $\ce$ to be the orbit category $\oo_{\sce}$ as in Section \ref{subgroup}. Because morphisms in $\oo_{\sce}$ are already $G$-maps, we can take the inclusion functor $\iota \colon \oo_{\sce} \hookrightarrow \lset{G}$ that sends $K \in \sce$ to the $G$-set $G/K$ and is constant on the morphisms. We write $E \oo_{\sce} := \gro{\iota}{\oo_{\sce}}$ for the Grothendieck construction.\\
 
 \item Consider the fusion category $\F_{\sce}$ as in Section \ref{subgroup} Take $v \colon \F_{\sce}^{\opp} \rarr \lset{G}$ as the functor that sends $K \in \sce$ to $G/C_{G}(K)$, and a morphism $c_{g,L,K} \in \F_{\sce}(L,K)$ to the $G$-map specified by
\begin{align*}
 G/C_{G}(K) &\rarr G/C_{G}(L) \\
 C_{G}(K) &\mapsto gC_{G}(L) \, .
\end{align*}
This $G$-map is well-defined, because ${^{g}L} \subseteq K$ implies $C_{G}(K) \subseteq C_{G}(^{g}L) = {^{g}C_{G}(L)}$. We write $E \A_{\sce} := \gro{v}{\F_{\sce}^{\opp}}$ for the Grothendieck construction. As a remark, Dwyer actually defines \cite[1.3, 3.1]{dw1} the centralizer decomposition as a Grothendieck construction over a different category $\A_{\sce}$ (from which the notation $E \A_{\sce}$ seems to come from). But Dwyer's $\A_{\sce}$ is actually equivalent to the fusion category $\F_{\sce}$: see Notbohm \cite[page 6]{notbohm-deco} for a proof.
\end{birki}

\begin{rem}[{\cite[($\dagger$)]{grodal-smith}}, {\cite[Proposition 2.14]{dw1}}]  \label{fix} 
First of all, using Remark \ref{stab-belli} and the definition of the functors $\oo_{\sce} \rarr \lset{G}$ and $\A_{\sce} \rarr \lset{G}$ used to construct $E \oo_{\sce}$ and $E \A_{\sce}$, we see that 
\begin{align*}
 \Iso_{G}(E \oo_{\sce}) = \sce \quad \text{and} \quad \Iso_{G}(E \A_{\sce}) = C_{G}(\sce) := \{C_{G}(H): H \in \sce \} \, .
\end{align*}
For the subgroup decomposition case, Proposition \ref{groth-genel} yields $E \oo_{\sce}^{H} \cong \sce_{\geq H}$. For the centralizer decomposition, the same proposition gives that there is a faithful functor $q \colon E \A_{\sce}^{H} \rarr \left( C_{G}(\sce)_{\geq H} \right)^{\opp}$, but $q$ is in general not full. Because there might be $K,L \in \sce$ for which $C_{G}(K) \geq C_{G}(L)$ without $K \leq L$. However, $q$ factors as
\begin{align*}
 \xymatrix{
 EA_{\sce}^{H} \ar[dr]^{q} \ar@{-->}[d]_{ p} \\
  \sce_{\leq C_{G}(H)}  \ar[r]_{C_{G}\,\,\,\,\,} & \left( C_{G}(\sce)_{\geq H} \right)^{\opp}}
\end{align*}
where $p$ is defined by $p(K,aC_{G}(K)) = {^{a}K}$ and $C_{G}$ is the order reversing map that sends a subgroup to its centralizer. As $E \A_{\sce}$ is a preorder, $p$ is automatically faithful. $p$ is also evidently surjective on objects, and (unlike $q$) $p$ is also full. As a result, we have an equivalence $E \A_{\sce}^{H} \cong \sce_{\leq C_{G}(H)}$ of categories.
\end{rem}

Using Theorem \ref{compute} and Remark \ref{fix}, we can now (usefully) expand the Lefschetz invariants of $E \oo_{\sce}$ and $E \A_{\sce}$ in both of the distinguished bases of the Burnside ring, proving Theorem \ref{lef-intro-EOC} and Theorem \ref{lef-intro-EAC} from the introduction. Recall that $\all{G}$ denotes the set of \textbf{all} subgroups of $G$, and $[G \bs \all{G}]$ is a set of representatives for the conjugacy classes of subgroups.
\begin{thm} \label{leflef}
Let $\sce$ be a set of subgroups of $G$ closed under conjugation. In the Burnside ring $\Omega(G)$, the expansion of the reduced series Lefschetz invariants of $E \oo_{\sce}$ and $E \A_{\sce}$ in the transitive $G$-sets are
\begin{align*}
\rlefs(E \oo_{\sce}) &= \sum_{H \in \sce} \frac{-\reu(\sce_{>H})}{|G:H|} \, [G/H] - [G/G] \, , \\
  \rlefs(E \A_{\sce}) &= \sum_{H \in \sce} \frac{-\reu(\sce_{<H})}{|G:C_{G}(H)|} \, [G/C_{G}(H)] - [G/G] \, .
\end{align*}
And their expansions in the primitive idempotents of $\Omega(G)$ are
\begin{align*}
  \rlefs(E\oo_{\sce}) = \sum_{\mathclap{\substack{K \in [G \bs \all{G}] \\ K \notin \sce}}} \, \reu(\sce_{> K}) \eps_{K} \, , \quad
  \rlefs(E \A_{\sce}) = \sum_{\mathclap{\substack{K \in [G \bs \all{G}] \\ C_{G}(K) \notin \sce}}} \, \reu(\sce_{< C_{G}(K)}) \eps_{K} \, .
\end{align*}
\end{thm}
\begin{proof}
 For the first set of equalities, we use Theorem \ref{compute}. The necessary skeletal weights were computed in Corollary \ref{orbit-weight} and Corollary \ref{fusion-coweight}, noting that a weighting on $\F_{\sce}^{\opp}$ is the same as a coweighting on $\F_{\sce}$. 
 
Let us also prove the idempotent expansion for $E \oo_{\sce}$, and leave the $E \A_{\sce}$ case out as it is similar. First of all, Corollary \ref{red-lef} yields
\begin{align*}
 \rlefs(E\oo_{\sce}) = \sum_{\mathclap{\substack{K \in [G \bs \all{G}]}}} \, \reus(E\oo_{\sce}^{K}) \eps_{K} \, .
\end{align*}
 At this point we would like to deduce $\reus(E \oo_{\sce}^{K}) = \reu(\sce_{\geq K})$ for any subgroup $K$. Although the categories $E \oo_{\sce}^{K}$ and $\sce_{\geq K}$ are equivalent by Remark \ref{fix}, the equality we want does not immediately follow. This is because the series Euler characteristic is \textbf{not} invariant under equivalences of categories, see \cite[Example 4.6]{berger}. But the category $E \oo_{\sce}^{K}$ is EI; thus it has skeletal \mob inversion (Example \ref{ex-ei}). Therefore $\reus(E \oo_{\sce}^{K}) = \reu(E \oo_{\sce}^{K})$ by Corollary \ref{series-ok}. Now Leinster's Euler characteristic $\reu$ \textbf{is} invariant under equivalences of categories \cite[Proposition 2.4]{leinster}, so we are good.

Finally, note that if $K \in \sce$, the poset $\sce_{\geq K}$ has $K$ as a unique minimal element and so $\reu(\sce_{\geq K}) = 0$. And if $K \notin \sce$, we have $\sce_{\geq K} = \sce_{> K}$.
\end{proof}

\section{Explicit induction formulae for Green functors} \label{green-section}


 Throughout this section, $A$ denotes a fixed \textbf{$\qq$-Green functor}. By this, we mean that 
\begin{birki}
 \item $A$ assigns to every subgroup $H$ an associative $\qq$-algebra $A(H)$ with identity $\one_{H} \in A(H)$, and 
 \item whenever $H \leq K$ are subgroups of $G$, there are $\qq$-linear maps $\res_{H}^{K} \colon A(K) \rarr A(H)$, $\ind_{H}^{K} \colon A(H) \rarr A(K)$ and $c_{g} \colon A(H) \rarr A(^{g} H)$ for every $g \in G$, satisfying axioms 1.1-1.9 in Th{\'e}venaz \cite{thev-remark}.
\end{birki}
In Th{\'e}venaz's notation, $r_{H}^{K}$ is our $\res_{H}^{K}$, $t_{H}^{K}$ is our $\ind_{H}^{K}$, and ${^{g}(-)}$ is our $c_{g}$ .

The Burnside functor, that assigns each subgroup $H \leq G$ to the Burnside ring $\Omega(H)$ is an example of a rational Green functor, which is initial among rational Green functors just like $\zz$ is initial among rings:

\begin{prop}[{\cite[Proposition 6.1]{thev-remark}}]\label{initial}
 There is a unique collection of $\qq$-algebra homomorphisms out of the Burnside rings $\{f_{H} \colon \Omega(H) \rarr A(H) \mid H \leq G \}$, which commute with restriction, induction and conjugation maps.
\end{prop}

In particular, chasing the identity element  $[H/H] \in \Omega(H)$ in the commutative diagram 
\begin{align*}
 \xymatrix{
 \Omega(H) \ar[r]^{\ind_{H}^{G}} \ar[d]_{f_{H}} & \Omega(G) \ar[d]^{f_{G}} \\
 A(H) \ar[r]^{\ind_{H}^{G}}  & A(G) }
\end{align*}
yields $f_{G}([G/H]) = \ind_{H}^{G}(\one_{H}) \in A(G)$.

\begin{defn} \label{primordial}
 We write $\prim{A}$ for the set of subgroups $H \leq G$ for which the $\qq$-linear map
\begin{align*}
 \mathop{\mathsmaller{\bigoplus}}_{K<H}\ind_{K}^{H} \colon \mathlarger{\bigoplus_{K < H}} A(K) \rarr A(H)
\end{align*}
is \textbf{not} surjective. The set $\prim{A}$ is called the \textbf{primordial set} of $A$, and if $H \in \prim{A}$, it is called a \textbf{primordial subgroup} of $A$.
\end{defn}
Note that $\prim{A}$ is closed under conjugation. There is an important vanishing property for subgroups outside $\prim{A}$:

\begin{prop}[{\cite[Proposition 6.4]{boltje-habil}}] \label{idemp-bye} If $K$ is \textbf{not} a primordial subgroup of $A$, then the canonical map $f_{G} \colon \Omega(G) \rarr A(G)$ of Proposition \ref{initial} sends the primitive idempotent $\eps_{K} \in \Omega(G)$ to $0 \in A(G)$. 
\end{prop}

Now we apply $f_{G}$ to the series Lefschetz invariants computed in Section \ref{sub-cen}. This results in an induction formula, which in this generality was first obtained by Th{\'e}venaz:

\begin{thm}[{\cite[Corollary 7.4]{thev-remark}}] \label{EOC-induction}
 Let $A$ be a $\qq$-Green functor. Suppose $\sce$ is a set of subgroups of $G$ closed under conjugation, such that $\sce$ contains the primordial subgroups of $A$. Then
\begin{align*}
 \one_{G} = \sum_{H \in \sce} \frac{-\reu(\sce_{>H})}{|G:H|} \ind_{H}^{G}(\one_{H})
\end{align*}
in $A(G)$.
\end{thm}
\begin{proof}
 Applying the ring homomorphism $f_{G} \colon \Omega(G) \rarr A(G)$ to the $G$-set expansion of $\LefS(E \oo_{\sce}) = \rlefs(E \oo_{\sce}) + [G/G]$ in Theorem \ref{leflef}, we get 
\begin{align*}
 f_{G}(\rlefs(E \oo_{\sce})) + \one_{G} =  \sum_{H \in \sce} \frac{-\reu(\sce_{>H})}{|G:H|} \ind_{H}^{G}(\one_{H})  \in A(G)\, .
\end{align*}
The idempotent expansion of the reduced invariant $\rlefs(E \oo_{\sce})$ in Theorem \ref{leflef} contains only $\eps_{K}$'s with $K$ outside $\sce$, hence outside $\prim{A}$. Thus by Proposition \ref{idemp-bye} it is mapped to zero under $f_{G}$.
\end{proof}

The novelty of our proof of Theorem \ref{EOC-induction} is that it shows the explicit induction formula ``comes from'' the subgroup decomposition category $E \oo_{\sce}$ in some sense. The same argument applied to the centralizer decomposition category $E \A_{\sce}$ yields a new induction formula.

\begin{thm} \label{EAC-induction}
  Let $A$ be a $\qq$-Green functor. Suppose $\sce$ is a set of subgroups of $G$ closed under conjugation, such that $\sce$ contains the \textbf{centralizer of} every primordial subgroup of $A$. Then
\begin{align*}
 \one_{G} = \sum_{H \in \sce} \frac{-\reu(\sce_{<H})}{|G:C_{G}(H)|}\ind_{C_{G}(H)}^{G} (\one_{C_{G}(H)})
\end{align*}
in $A(G)$.
\end{thm}
\begin{proof}
 The proof is analogous to Theorem \ref{EOC-induction}. Use Theorem \ref{leflef} and observe that if $\sce$ contains the centralizers of subgroups in $\prim{A}$, then the idempotent expansion of the reduced series Lefschetz invariant $\rlefs(E \A_{\sce})$ contains only $\eps_{K}$'s with $C_{G}(K) \notin \sce$, and hence with $K \notin \prim{A}$. Now use Proposition \ref{idemp-bye}.
\end{proof}
Observe that taking take $\sce$ to be exactly the set of centralizers of subgroups in $\prim{A}$, the subgroups we are inducing up are the centralizers of those in $\sce$, hence the \textbf{double centralizers} of subgroups in $\prim{A}$. See Example \ref{work-it} for a worked out example. Because of this double centralizer phenomenon, the induction formula in Theorem \ref{EAC-induction} is not as optimal as the one in Theorem \ref{EOC-induction} in the sense that we might be inducing from bigger subgroups than what is sufficient for $A$. On the other hand, this may result in smaller indices in the denominators and hence a more integral formula. A second issue is that while Theorem \ref{EOC-induction} yields a non-trivial induction formula as long as $G \notin \prim{A}$, the formula in Theorem \ref{EAC-induction} becomes void if $\prim{A}$ contains a subgroup with trivial centralizer.

\subsection{Applications to representations, cohomology, and topology} \label{conseq}
Let $R$ be a unital commutative ring, $G$ a finite group, and $a_{R}(G)$ be the representation ring of finitely generated $RG$-modules. More precisely, the set of isomorphism classes of finitely generated $RG$-modules forms a commutative semiring under direct sum and tensor product, for which $a_{R}(G)$ is the associated Grothendieck ring. The assignment $H \mapsto a_{R}(H)$ defines a $\zz$-Green functor, and hence $A_{R} := \qq \otimes_{\zz} a_{R}$ is a $\qq$-Green functor. The primordial subgroups for a general $R$ was worked out by Dress:

\begin{thm}[{\cite[Theorem $1'$, Theorem 2]{dress-on-integral}}] \label{dress-prim} A subgroup $H \leq G$ is a primordial subgroup of $A_{R}$ if and only if one of the following holds: 
\begin{birki}
 \item $H$ is cyclic. 
 \item There exists a prime $p$ with $pR \neq R$ such that $H /O_{p}(H)$ is cyclic.
\end{birki}
\end{thm}

It is now a matter of bringing the threads together to prove the promised Theorem \ref{new-induction}.

\begin{proof}[Proof of \emph{\textbf{Theorem \ref{new-induction}}}] Noting that the multiplicative identity in $A_{R}(H)$ is the trivial representation $R$, apply Theorem \ref{EAC-induction} to the Green functor $A_{R}$, using Theorem \ref{dress-prim}.
\end{proof}

\begin{proof}[Proof of \emph{\textbf{Theorem \ref{new-coho}}}]
The assignment $L \mapsto \Ext^{k}_{RG}(L,M)$ defines a linear map $A_{R}(G) \rarr A_{R}(1)$. Apply this map to the equality in Theorem \ref{new-induction}, using a form of Shapiro's lemma that gives 
\begin{align*}
  \Ext^{k}_{RG}(\ind_{H}^{G}(R),M) \cong \Ext^{k}_{RH}(R, \res_{H}^{G}(M)) = \co^{k}(H;M) \, 
\end{align*}
for any subgroup $H \leq G$. We can use $\Tor$ to get a similar formula in group homology, and use Tate Ext groups for Tate cohomology.
\end{proof}

\begin{que}
Is it possible to avoid using Dress's result to prove Theorem \ref{webb-formula} and Theorem \ref{new-induction} by working directly with the chain complexes of $RG$-modules associated to $E\oo_{\sce}$ and $E\A_{\sce}$? What we are lacking here is a chain-level reason for the divergent alternating sum of modules in an \textbf{unbounded} (from one side) chain complex to vanish. On the other hand, there is an obvious condition for \textbf{bounded} chain complexes: a chain homotopy equivalence with the zero complex (see \cite[Proposition 0.3]{brown-coho-book}). This is not enough for the infinite case, as can be seen from 
\begin{align*}
 \cdots \rarr R \rarr R \rarr R \rarr R \rarr 0
\end{align*}
where the maps alternate between the identity and zero maps. The divergent alternating sum would yield $\frac{1}{2}R$ here, not zero.
\end{que}

\begin{proof}[Proof of \emph{\textbf{Theorem \ref{bg-split-cent}}}]
With $R = \compl{\zz}$, the equality in the statement of Theorem \ref{new-induction} (after clearing the denominators etc.) can be written as an isomorphism $\compl{\zz}S \cong \compl{\zz}T$ of permutation $\compl{\zz}G$-modules for certain $G$-sets $S,T$, noting $\ind_{K}^{G}(\compl{\zz}) = \compl{\zz}[G/K]$. We then also get $\ff_{p}S \cong \ff_{p}T$ by mod-$p$ reduction. Minami shows \cite[Lemma 6.8]{minami-splitting} that then for any free $G$-space $X$ we have an equivalence
\begin{align*}
 \complet{\Sigma^{\infty}X \times_{G} S} \simeq \complet{\Sigma^{\infty}X \times_{G} T}
\end{align*}
of spectra. This can be turned back into a fractional expression, namely
\begin{align*}
 \complet{\Sigma^{\infty}X/G} \simeq \bigvee_{H \in \sce} \frac{-\mu_{\sce_{-}}(-\infty,H)}{|G:C_{G}(H)|} \complet{\Sigma^{\infty}X/C_{G}(H)} \, ,
\end{align*}
noting that $X \times_{G} G/K \simeq X/K$. Taking $X = \text{E}G$ yields the desired result.
\end{proof}

\subsection{Canonicity and non-canonicity of induction formulae} \label{canon} A natural question with an explicit induction formula is whether it is compatible with the restriction maps. Let us expand on what this means: using an explicit induction formula for $A$, we get an expression of the form 
\begin{align*}
 \one_{G} = \sum_{H \leq G} \lambda_{H} \ind_{H}^{G}(\one_{H}) \in A(G)\, .
\end{align*}
Given a subgroup $K \leq G$, if we apply $\res_{K}^{G} \colon A(G) \rarr A(K)$ to both sides, we get 
\begin{align*}
 \one_{K} &= \sum_{H \leq G} \lambda_{H} \res_{K}^{G}(\ind_{H}^{G}(\one_{H})) = \sum_{H \leq G} \lambda_{H} \cdot \sum_{g \in [K/G \bs H]} \ind_{K \cap {^{g}H}}^{K} (\one_{K \cap {^{g} H}})
\end{align*}
by the Mackey axiom \cite[1.5]{thev-remark}. Collecting like terms, we would get an expression 
\begin{align*}
 \one_{K} = \sum_{L \leq K} \gamma_{L} \ind_{L}^{K}(\one_{L}) \in A(K) \, .
\end{align*}
On the other hand, $A$ restricted to the subgroups of $K$ is a perfectly valid Green functor for the group $K$. Let us write $A|_{K}$ for this Green functor. Now we could apply the induction formula at hand directly to $A|_{K}$ and get another expression for $\one_{K}$ like above. The question is, would the coefficients that appear here agree with the $\gamma_{L}$'s above? Boltje carried out a detailed analysis of such restriction-respecting formulae (which we shall call \textbf{canonical}, following him) in great generality; see \cite{boltje-habil} and \cite{boltje-ja}. A consequence of his analysis  for a Green functor $A$ as defined in the beginning of Section \ref{green-section} is the following: not only the induction formula in Theorem \ref{EOC-induction} with $\sce = \prim{A}$ is canonical, but also it is minimal in a precise sense among all other canonical induction formulae for $A$; see \cite[Example 2.8]{boltje-habil}. 

We point out an elementary way of seeing the canonicity when $\sce$ in Theorem \ref{EOC-induction} is closed under taking subgroups. 

\begin{prop} \label{respect}
 Let $\sce$ be a set of subgroups of $G$ that is closed under conjugation \textbf{and} taking subgroups. For every subgroup $K \leq G$, write $\sce(K):= \{H \leq K: H \in \sce\}$, so we have elements $\LefS(E \oo_{\sce(K)}) \in \Omega(K)$. Let $T$ be an indeterminate. The evaluation maps $\{s_{K} : K \leq G\}$ out of the polynomial algebra $\qq[T]$ defined by
\begin{align*}
 s_{K} \colon \qq[T] &\mapsto \Omega(K) \\
 T &\mapsto \LefS(E \oo_{\sce(K)})
\end{align*}
are compatible with restriction and conjugation maps on the Burnside ring. That is, $\res_{H}^{K} \circ s_{K} = s_{H}$ whenever $H \leq K$ and $c_{g} \circ s_{H} = s_{gHg^{-1}}$ for every $g \in G$.
\end{prop}
\begin{proof}
By Proposition \ref{lef-respect} it is enough to show 
\begin{align*}
 \LefS(E \oo_{\sce(K)}) = \LefS(\Res_{K}^{G}(E \oo_{\sce})) \in \Omega(K)
\end{align*}
for every subgroup $K$. And to see the $K$-categories $E \oo_{\sce(K)}$ and $\Res_{K}^{G}(E \oo_{\sce})$ have the same series Lefschetz invariants in $\Omega(K)$, by Proposition \ref{lef-eu} it is enough to show 
$
 \eus(E \oo_{\sce(K)}^{H})  = \eus(E \oo_{\sce}^{H})
$ for every $H \leq K$. By Remark \ref{fix}, this amounts to checking $\chi(\sce(K)_{\geq H}) = \chi(\sce_{\geq H})$. Now if $H \in \sce$, both $\sce(K)_{\geq H}$ and $\sce_{\geq H}$ have a unique minimal element, namely $H$. And if $H \notin \sce$, we have $\sce(K)_{\geq H} = \sce_{\geq H} = \empt$ because $\sce$ is assumed to be closed under taking subgroups.
\end{proof}
The canonicity of Theorem \ref{EOC-induction} with taking $\sce$ to be the subgroup-closure of $\prim{A}$, which is the so-called \textbf{defect base} of $A$, follows immediately because $\prim{A|_{K}} = \prim{A}(K)$ by \cite[Proposition 2.3]{thev-remark}. In several applications $\prim{A}$ is already subgroup-closed.
\begin{rem}
 In the proof of Proposition \ref{respect}, we see that when $\sce$ is closed under taking subgroups, $E\oo_{\sce}^{H}$ is 
\begin{birki}
 \item contractible if $H \in \sce$, and
 \item empty if $H \notin \sce$.
\end{birki}
These conditions imply that $E \oo_{\sce}$ is a model for the \textbf{classifying space for $\sce$} \cite[Definition 1.8, Theorem 1.9]{luck-classifying}. Its series Lefschetz invariant reflects this with its multiplicative property, for
\begin{align*}
 \LefS(E \oo_{\sce}) = \sum_{H \in [G \bs \sce]} \eps_{H} =: \eps_{\sce}
\end{align*}
(use Proposition \ref{lef-eu}) is exactly the idempotent associated to $\sce$ in the Burnside ring $\Omega(G)$.

We also see that for any subgroup $K \leq G$, not only $\res_{K}^{G}(E \oo_{\sce})$ and $E \oo_{\sce(K)}$ have the same series Lefschetz invariant in $B(K)$ as shown in Proposition \ref{respect}, but also the same $K$-homotopy type. 
\end{rem}
Unlike the subgroup decomposition category, the formula coming from the centralizer decomposition category $E \A_{\sce}$ in Theorem \ref{EAC-induction} is \textbf{not} canonical, at least when $\sce$ is minimally chosen as the set of centralizers of the primordial subgroups. We illustrate this in an example:
\begin{ex} \label{work-it}
  Let $G = S_{4}$ and consider the Green functor
\begin{align*}
  A_{\cc} \colon H \mapsto \qq \otimes_{\zz}  \{\text{ring of complex $H$-characters}\}  \, .
\end{align*}
Then $\prim{A_{\cc}}$ is the set of cyclic subgroups of $G$: the forward inclusion here is Artin's induction theorem. To exhaust $\prim{A_{\cc}}$ up to $G$-conjugacy, set $C_{2}' := \gen{(12)}$, $C_{2}'' := \gen{(12)(34)}$, $C_{3} := \gen{(123)}$, $C_{4} := \gen{(1234)}$, and $1$ to be the trivial subgroup. Here $C_{3}$ and $C_{4}$ are self-centralizing, whereas $V_{4}' := C_{G}(C_{2}') = \gen{(12),(34)}$ and $D_{8} := C_{G}(C_{2}'') = \gen{(12),(1324)}$ and of course $G = C_{G}(1)$. Now Theorem \ref{EAC-induction} applies to the union of the $G$-conjugacy classes
$
 \sce := [C_{3}] \cup [C_{4}] \cup [V_{4}'] \cup [D_{8}] \cup [G]
$.
In other words, $\sce$ is the set of centralizers of cyclic subgroups of $G$. Below is a picture of the poset $\sce_{-} = \sce \sqcup \{-\infty\}$ :
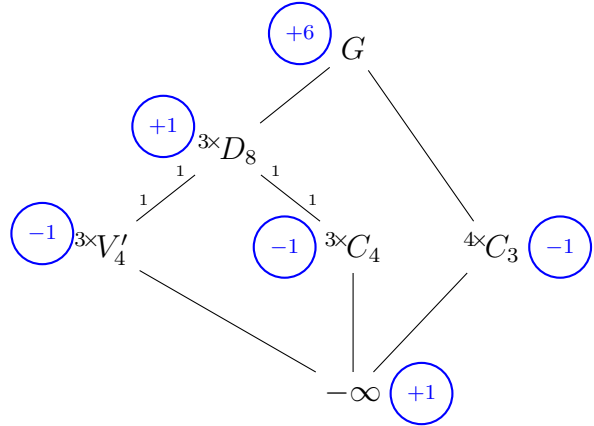
\begin{center}
\begin{tikzpicture}[node distance =2cm]
	\node(G) {$G$};
	\node(D8) [below left=0.7cm and 0.8cm of G] {$^{3\! \times}\!D_{8}$};
	\node(V4) [below left = 2cm and 2.5cm of G] {$^{3\! \times}\! V_{4}'$};
	\node(C4) [below = 2cm of G] {$^{3\! \times}\!C_{4}$};
	\node(C3) [below right = 2.0cm and 1cm of G] {$^{4\! \times}\! C_{3}$};
	\node(1) [below = 4cm of G] {$-\infty$};
	\draw (G)--(D8);
	\draw (G)--(C3);
	\draw (D8) -- node[near start, above]{$\scriptscriptstyle{1}$} node[very near end, above]{$\scriptscriptstyle{1}$} (V4);
	\draw (D8)-- node[near start, above]{$\scriptscriptstyle{1}$} node [very near end, above]{$\scriptscriptstyle{1}$} (C4);
	\draw(1)--(V4);
	\draw(1)--(C4);
	\draw(1)--(C3);
	\draw[blue] node[thick, on grid, above left = 2mm and 7 mm of G, circle,draw]{$\scriptstyle{+6}$};
	\draw[blue] node[thick, on grid, above left = 3mm and 8.5mm of D8, circle,draw]{$\scriptstyle{+1}$};
	\draw[blue] node[thick, on grid, above left = 2mm and 7.8 mm of V4, circle,draw]{$\scriptstyle{-1}$};
	\draw[blue] node[thick, on grid, left = 9mm of C4, circle, draw]{$\scriptstyle{-1}$};
	\draw[blue] node[thick, on grid, right = 9mm of C3, circle, draw]{$\scriptstyle{-1}$};
	\draw[blue] node[thick, on grid, right = 9mm of 1, circle, draw]{$\scriptstyle{+1}$};
\end{tikzpicture}
\end{center}
Here, the notation ${^{3 \! \times} \! D_{8}}$ means that $D_{8}$ has $3$ conjugates in $G$. The edge connecting $D_{8}$ to $V_{4}'$ having two $1$'s  means that each conjugate of $V_{4}'$ is contained in exactly 1 conjugate of $D_{8}$ in $G$, etc. The numbers in circles record the \mob function values $\mu_{\sce_{-}}(-\infty, H) = \reu(\sce_{<H})$ for $H \in \sce$. The second round of centralizers go $C_{G}(G) = 1, C_{G}(D_{8}) =C_{2}''$, and $V_{4}', C_{3}$,$C_{4}$ are self-centralizing. Writing $\cc[G/H] = \ind_{H}^{G}(\one_{H}) \in A_{\cc}(G) $ for the complex permutation representation of the $G$-set $G/H$, Theorem \ref{EAC-induction} yields
\begin{align*}
\cc[G/G] = \frac{-6}{24} \cc[G/1] + 3 \cdot \frac{-1}{12} \cc[G/C_{2}''] + 3 \cdot \frac{1}{6} \cc[G/V_{4}'] + 3 \cdot \frac{1}{6} \cc[G/C_{4}] + 4 \cdot \frac{1}{8} \cc[G/C_{3}] \, ,
\end{align*}
which may also be verified by checking the character values. Now, restricting to the alternating group $A_{4}$ and applying the Mackey double coset formula several times, the above formula for $\cc[G/G]$ restricts to
\begin{align*}
 \cc[A_{4}/A_{4}] = \frac{-1}{2} \cc[A_{4}/1] + \frac{1}{2}\cc[A_{4}/C_{2}''] + \cc[A_{4}/C_{3}] \in A_{\cc}(A_{4}) \, .
\end{align*}
To compare, let us apply Theorem \ref{EAC-induction} directly to $A_{4}$ and centralizers of the cyclic subgroups of $A_{4}$. Up to $A_{4}$-conjugacy, $1$,$C_{3}$ and $C_{2}''$ are the only cyclic subgroups in $A_{4}$. $C_{3}$ is self centralizing in $A_{4}$, whereas $V_{4}'' := C_{A_{4}}(C_{2}'') = \{(), (12)(34), (13)(24), (14)(23)\}$ and of course $C_{A_{4}}(1) = A_{4}$. Taking $\sce'$ to be the $A_{4}$-conjugates of $C_{3}$ and $V_{4}''$, the poset $\sce'_{-} = \sce' \sqcup \{-\infty\}$ looks like:
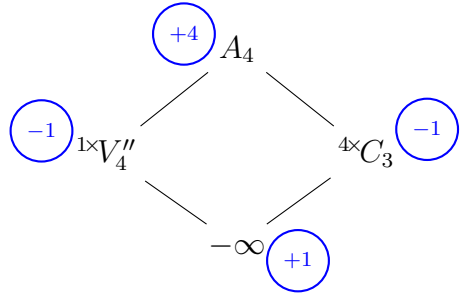
\begin{center}
\begin{tikzpicture}[node distance =2cm]
	\node(G) {$A_{4}$};
	\node(V4) [below left=0.7cm and 0.8cm of G] {$^{1\! \times}\!V_{4}''$};
	\node(C3) [below right = 0.7cm and 0.8cm of G] {$^{4\! \times}\! C_{3}$};
	\node(1) [below = 2cm of G] {$-\infty$};
	\draw (G)--(V4);
	\draw (G)--(C3);
	\draw(1)--(V4);
	\draw(1)--(C3);
	\draw[blue] node[thick, on grid, above left = 2mm and 7 mm of G, circle,draw]{$\scriptstyle{+4}$};
	\draw[blue] node[thick, on grid, above left = 3mm and 8.5mm of V4, circle,draw]{$\scriptstyle{-1}$};
	\draw[blue] node[thick, on grid, above right = 3mm and 8mm of C3, circle, draw]{$\scriptstyle{-1}$};
	\draw[blue] node[thick, on grid, below right = 2mm and 8mm of 1, circle, draw]{$\scriptstyle{+1}$};
\end{tikzpicture}
\end{center}
Noting that both $V_{4}''$ and $C_{3}$ are self-centralizing in $A_{4}$ and $C_{A_{4}}(A_{4}) = 1$, Theorem \ref{EAC-induction} applied to $A_{4}$ and $\sce'$ yields 
\begin{align*}
 \cc[A_{4}/A_{4}] &= \frac{-4}{12} \cc[A_{4}/1] + \frac{1}{3} \cc[A_{4}/V_{4}''] + 4 \cdot \frac{1}{4} \cc[A_{4}/C_{3}] \, ,
\end{align*}
a different formula than what we obtained above by restricting the formula for $G=S_{4}$.
\end{ex}
  
\bibliographystyle{amsalpha}
\bibliography{ulan}

\end{document}